\documentclass[a4paper]{amsart}
\usepackage[utf8]{inputenc}
\usepackage[T1]{fontenc}

\usepackage{amssymb}
\usepackage{mathrsfs}
\usepackage{array}
\usepackage{bbm}
\usepackage{stmaryrd}
\usepackage[shortlabels]{enumitem}

\usepackage{slashed}

\usepackage{csquotes}

\usepackage{mathtools}
\mathtoolsset{showonlyrefs,showmanualtags}

\usepackage[svgnames]{xcolor}
\usepackage[unicode,plainpages=false]{hyperref}

\usepackage{enumitem}

\newcommand{\m}[1]{
\ifdefequal{#1}{1}
{\mathbbm{#1}}
{\mathbb{#1}}
}

\newcommand{\q}[1]{\mathcal{#1}}

\newcommand{\ds}{\displaystyle}
\newcommand{\e}{\varepsilon}

\newcommand\R{\mathbb{R}}
\newcommand{\ext}{\textnormal{ext}}
\newcommand\nor[2]{\left\|#1\right\|_{#2}}

\newcommand\Td{\mathcal{T}}
\newcommand\Sdmu{\mathbb{S}^{d-1}}
\newcommand{\prnl}{\mathcal{P}(R)}

\renewcommand{\le}{\leqslant}
\renewcommand{\ge}{\geqslant}
\newcommand\M{\mathbb{M}}

\theoremstyle{plain}
\newtheorem{thm}{Theorem}[section]
\newtheorem*{thm*}{Theorem}

\newtheorem{prop}[thm]{Proposition}
\newtheorem{cor}[thm]{Corollary}
\newtheorem{lem}[thm]{Lemma}

\theoremstyle{definition}
\newtheorem{definition}[thm]{Definition}

\theoremstyle{remark}
\newtheorem{remark}[thm]{Remark}

\usepackage[svgnames]{xcolor}

\newcommand\Nd{\mathcal{N}}
\DeclareMathOperator{\Span}{\mathrm{Span}}

\date{}

\title{On the set of non radiative solutions for the energy critical wave equation}

\author{Raphaël \textsc{Côte}}
\address{Institut de Recherche Mathématique Avancée, UMR 7501, Université de Strasbourg, 7 rue René-Descartes, F-67084 Strasbourg Cedex, France}
\email{rcote@unistra.fr}

\author[C. Laurent]{Camille Laurent}
\address{CNRS UMR 9008, Universit\'e Reims-Champagne-Ardennes, Laboratoire de Math\'ematiques de Reims (LMR), Moulin de la Housse-BP 1039, 51687 REIMS cedex 2, France}

\email{camille.laurent@univ-reims.fr} 

\thanks{RC acknowledges support from the University of Strasbourg Institute for Advanced Study (USIAS) for a Fellowship within the French national programme ``Investment for the future'' (IdEx-Unistra).}

\allowdisplaybreaks

\begin{document}

\maketitle

\begin{abstract}
    Non radiative solutions of the energy critical non linear wave equation are global solutions $u$ that furthermore have vanishing asymptotic energy outside the lightcone at both $t \to \pm \infty$:
    \[ \lim_{t \to \pm \infty} \| \nabla_{t,x} u(t) \|_{L^2(|x| \ge |t|+R)} = 0, \]
    for some $R>0$.
    They were shown to play an important role in the analysis of long time dynamics of solutions, in particular regarding the soliton resolution: we refer to the seminal works of Duyckaerts, Kenig and Merle, see \cite{DKM:23} and the references therein.
    
    We show that the set of non radiative solutions which are small in the energy space is a manifold whose tangent space at $0$ is given by non radiative solutions to the linear equation (described in \cite{CL24}). We also construct nonlinear solutions with an arbitrary prescribed radiation field.
\end{abstract}
\section{Introduction}

 We consider solutions $u : I \times \m R^d \to \m R$ ($I$ interval of $\m R$) of the energy critical semilinear wave equation in dimension $3 \le d \le 6$:
\begin{equation} \label{eq:nlw}
\Box u = f(u), 
\end{equation}
with $f(x) = \pm |x|^{q-1} x$ or $f(x) = \pm x^q$ (if $q$ is an integer), where $q = \frac{d+2}{d-2}$ is the $\dot{H}^1$-critical exponent. If $u$ is a time dependent function, we denote $\vec u = (u, \partial_t u)$.

Denote $\mathcal{H} := \dot{H}^1(\m R^{d})\times L^{2}(\m R^{d})$. For a time interval $I\subset \m R$, we define the spaces
\begin{equation} \label{def:spaces_WN}
     W(I) =  L^{q}(I,L^{2q}(\m R^{d})) \quad \text{ and} \quad N(I)= L^1(I,L^2(\m R^d)) 
\end{equation}
together with
\begin{equation} \label{def:space_X}
X(I)= \{ u \in \mathscr C (I,\dot{H}^{1}(\m R^d)) \cap W(I) :  \partial_t u \in \mathscr C(I, L^{2}(\m R^d)) \},
\end{equation}
with the natural norm
\[ \| u \|_{X(I)} = \| u \|_{\mathscr C(I, \dot H^1(\m R^d))} + \| \partial_t u \|_{\mathscr C(I, L^2(\m R^d))} + \| u \|_{W(I)}. \]

We now define the linear and nonlinear flows: if $(u_0,u_1) \in \q H$, then $ \vec u_L(t) = S_L(t)(u_0,u_1)$
 is the solution of the linear wave equation
\begin{equation} \label{eq:lw}
\begin{cases}
\Box u_L = 0, \\
\vec u_L(0) = (u_0,u_1).
\end{cases}
\end{equation}
Similarly, concerning the nonlinear equation, the problem is locally well posed for data $(u_0,u_1)$ in $\q H$ and furthermore, if they are small in that space, the non linear solution is global and scatters linearly as $t \to \pm \infty$: see for example Strauss \cite{S:68}, Rauch \cite{R:79}, Pecher \cite{Pec84}, Ginibre-Velo \cite{GV:95} or Lindblad-Sogge \cite{LS:95} among others. In that case, we will denote $\vec u(t) = \q S(t)(u_0,u_1)$ the solution  to the nonlinear wave equation \eqref{eq:nlw} with initial data $\vec u(0) = (u_0,u_1)$. We may write $S_L(u_0,u_1)$ and $\q S(u_0,u_1)$ to denote the space time function $\vec u_L$ and $\vec u$ respectively. 

\bigskip

For a space time function $\vec v \in X(\m R)$, we define its radiation energy outside a light cone (of base $R \ge 0$) by
\begin{align*}
E_{\textnormal{ext},R}(\vec v)& := \frac{1}{2} \left( \lim_{t \to +\infty} ( \nor{\nabla v}{L^2(|x| \ge t+R)}^2 +  \nor{\partial_t v}{L^2(|x| \ge t+R)}^2) \right.  \\
& \qquad \qquad \left. + \lim_{t \to -\infty} (\nor{\nabla v}{L^2(|x| \ge |t|+R)}^2 +  \nor{\partial_t v}{L^2(|x| \ge |t|+R)}^2) \right),
\end{align*}
provided that the limits exist.

If $u$ is a solution to the linear energy critical wave equation \eqref{eq:lw}, due to finite speed of propagation, the energy outside a light cone
\[  \nor{\nabla u}{L^2(|x| \ge t+R)}^2 +  \nor{\partial_t u}{L^2(|x| \ge t+R)}^2 \]
is decreasing as a function of $t\geq 0$ and admits a limit as $t \to +\infty$, for any $R \ge 0$ (and also as $t \to -\infty$), and so its radiation energy is well defined for any $R \ge 0$. If $u \in X(\m R)$ is a global solution to the non linear energy critical wave equation \eqref{eq:lw}, there is linear scattering, so that the radiative energy is well defined as well (this is always the case for small data). See also \cite{DKM:19} for the case of large data, global solutions to \eqref{eq:lw} which enjoy an a priori  $\dot H^1 \times L^2$ bound.

\bigskip

We say that a space time function $\vec v \in X(\m R)$ is \emph{non radiative} if $E_{\ext,R}(\vec v) =0$ for some $R \ge 0$. Non radiative solutions play a crucial role as the main obstruction in the energy channel method: this machinery was developed with great success, by Duyckaerts, Kenig, Merle and collaborators, to understand the long time behavior of solution to the \emph{radial} energy critical non linear wave equation, in relation with the soliton resolution conjecture. The classification of small \emph{radial} non radiative solutions of \eqref{eq:nlw} was addressed in \cite{DKM:21,CDKM:23}; we also refer \cite{DuyKMClassif,DKM:23} and the references therein. An important issue, raised in \cite[Theorem 1.3]{CDKM:23}, is to  extend solutions which are non radiative for some $R >0$ (defined on the exterior light cone $\{ |x| \ge t+R \}$), to solutions which are non radiative for $R =0$. All in all, we believe that a fine understanding of these particular solutions might constitute a useful step as well in the soliton resolution in the general case (without symmetry).

\bigskip

Our goal in this article is to give a description of an initial data which leads to non radiative solutions $\vec u$ to \eqref{eq:nlw}.

\bigskip

We described in \cite{CL24}, for odd dimensions, the linear space $P_L(R)$ of initial datum $(v_0,v_1)\in \q H$ that give rise to a solution $\vec v = S_L(v_0,v_1)$ to the linear wave equation such that $E_{\ext,R} (\vec v) =0$, in terms of the the Radon transform of the initial data $(v_0,v_1)$ and according to its decomposition in spherical harmonics: for the convenience of the reader, we give further details in the Appendix \ref{appendix1}, see in particular \eqref{PRHarmonic}. This was first done for radial data in odd dimension by \cite{KLLS15}, and in even dimension in \cite{LSW24} (see also \cite{CoteKenigSchlagEquipart}), and it was extended to non radial data for odd dimensions in \cite{CL24} and later in even dimension in \cite{LSWW:22}.

Let us define the operator $\Td$ as follows: for a function $v$ defined on $\m R^d$, $\Td v$ is a function of two variables $(s,\omega)$, defined on $\m R \times \m S^{d-1}$ by its (partial) Fourier transform in the first variable $s$:
\begin{gather}
\label{def:T}
\q F_{s \to \nu}  (\Td v)(\nu, \omega) = c_0  |\nu|^{\frac{d-1}{2}} (e^{i \tau}  \m1 _{\nu < 0}   + e^{-i\tau} \m 1_{\nu \ge 0} ) \hat v(\nu \omega), \\
 \label{def:tau}
\text{where} \quad \tau :=  \frac{d-1}{4} \pi, \quad \quad c_0 = \frac{1}{\sqrt{2(2\pi)^{d-1}}},
\end{gather}
and $\hat v$ is the Fourier transform on $\m R^d$ of $v$:
\[ \forall \xi \in \m R^d, \quad \hat v(\xi) = \int e^{-i \xi \cdot x} v(x) dx. \]

The previous formula can also be expressed in term of the Radon transform $\q R$: it is defined for $f \in \q S(\m R^d)$ by
\[ \q Rf : \m R \times \m S^{d-1} \to \m R, \quad (s,\omega) \mapsto \int_{\omega \cdot y =s} f(y) dy, \]
(with Lebesgue measure on the hyperplane $\{ y : \omega \cdot y =s \}$) and it can be extended to  $f \in L^ 2(\m R^d)$. Then there holds
\[ \q T = m_{d}(D_{s})\q R \quad \text{where} \quad m_{d}(\nu) :=c_{0} |\nu|^{\frac{d-1}{2}}\left(e^{i\tau }\m 1_{ \nu<0} +e^{-i\tau }\m 1_{ \nu \ge 0} \right). \]
In odd dimension, this relation simply writes:
\[ \q T = \frac{(-1)^{\frac{d-1}{2}}}{\sqrt{2(2\pi)^{d-1}}} \partial_s^{ \frac{d-1}{2}} \q R. \]
We refer to \cite{CL24} for details.

Our statement regarding the radiation of linear wave solutions is as follows. It is closely related to the radiation field of Friedlander \cite{F:80}, see also Katayama \cite{Kata:13} for a related result.

\begin{prop}[Radiation field and concentration of energy on the light cone, {\cite[Theorem 1.1]{CL24}}] \label{th1}
{\ }
Let $(v_0, v_1) \in \q H$, and $\vec v = S_L(v_0,v_1)$ be the linear solution to \eqref{eq:lw}. Then as $t \to +\infty$, there holds the convergence in $L^2(\m R^d)^{1+d}$
\begin{gather} \label{eq:expansion_w_1}
\nabla_{t,x} v(t,x) - \frac{1}{\sqrt{2}|x|^{\frac{d-1}{2}}}(\partial_s \Td v_0 - \Td v_1) \left(|x|-t, \frac{x}{|x|} \right) \times \begin{pmatrix}
-1 \\
x/|x|
\end{pmatrix}  \to 0.
\end{gather}
Furthermore, one has
\begin{align}
\lim_{t \to +\infty} \nor{\nabla v}{L^2(|x| \ge t + R)}^2 = \lim_{t \to +\infty} \nor{\partial_{t} v}{L^2(|x| \ge t + R)}^2  = \frac{1}{2} \|  \partial_{s}\Td v_{0} - \Td v_{1} \|_{L^2([R,+\infty) \times \Sdmu)}^2.  \label{equivfinallemmaH1_2} 
\end{align}
\end{prop}

The function $\partial_s \Td v_0 - \Td v_1$ in \eqref{eq:expansion_w_1} is called the radiation field (at $+\infty$) of $\vec v$. Note that changing $v_{1}$ to $-v_{1}$ and reversing time, we get the same result in negative time
\begin{align}
\lim_{t \to -\infty} \nor{\nabla v}{L^2(|x| \ge t + R)}^2 = \lim_{t \to -\infty} \nor{\partial_{t} v}{L^2(|x| \ge t + R)}^2  = \frac{1}{2} \|  \partial_{s}\Td v_{0} + \Td v_{1} \|_{L^2([R,+\infty) \times \Sdmu)}^2, \label{equivfinallemmaH1_2-infty} 
\end{align}
so that
\begin{equation}
 E_{\ext,R}(\vec v) = \|  \partial_{s}\Td v_{0}  \|_{L^2([R,+\infty) \times \Sdmu)}^2 + \|  \Td v_{1} \|_{L^2([R,+\infty) \times \Sdmu)}^2.
\end{equation}

We want to define $\prnl$ the (nonlinear) space of initial datum giving rise to nonlinear radiative solutions. More precisely, we denote 
\begin{align*}
\prnl=\left\{(u_{0},u_{1}): \q{S}(u_{0},u_{1})
\textnormal{ is defined globally on } \m R \text{ and } E_{\ext,R}(\q{S}(u_{0},u_{1}))=0\right\}.
\end{align*}
Our first result states that around $0 \in \q H$, $\prnl$ is a submanifold of $\q H$, whose tangent space at $0$  is  $P_L(R)$.

\begin{thm} \label{thm:1}
Let $d =3$ or $5$. Let $R>0$, and denote $\pi_R$ the orthogonal projection on $P_L(R)$ (in $\q H$).

There exists $\e>0$ and a $\mathscr C^1$ map
\[ \Phi:  B_{\q H}(0,\e) \to  \q H.  \]
so that $\Phi$ is a diffeomorphism to its image $V=\Phi(B_{\q H}(0,\e))\subset B_{\q H}(0,2\e)$ whose differential at zero is the identity and satisfies 
 \[ \forall (u_0,u_1) \in   B_{\q H}(0,\e), \quad \nor{ (u_{0},u_{1})- \Phi(u_{0},u_{1})}{\q H} \le \| (u_0,u_1) \|_{\q H}^{q}. \]
  \[ \forall (u_0,u_1) \in   B_{\q H}(0,\e), \quad  \pi_R \circ \Phi(u_{0},u_{1})=\pi_{R}(u_{0},u_{1}). \]
Moreover, when restricted to $P_L(R)$, we have $\Phi\left(P_L(R)\cap B_{\q H}(0,\e)\right)=\prnl\cap V $.

In particular, $\prnl\cap V$ is a submanifold of $\q H$ with tangent space at $0$ equal to $P_L(R)$. Moreover, $(\pi_{R})_{\left|\prnl\cap V\right.}$ is a chart from $\prnl\cap V$ to $P_L(R)\cap B_{\q H}(0,\e)$ with inverse $\Phi$.

\end{thm}
In particular, this result proves that there are a lot of nonlinear radiative solutions, at least as many as the linear set $P_L(R)$ which is actually a large space, see Appendix~\ref{appendix1}. This description extends some results of \cite{CDKM:23} to the non radial case.

\medskip

Simple non radiative solutions can be constructed as follows: it suffices to consider a static solutions $u(t,x)=u(x)$ for $|x|\geq |t|+R$, with $-\Delta u=f(u)$. Such solutions outside of a ball have been precisely described in our recent work \cite{CL:24} for analytic nonlinearity (which is useful for \eqref{eq:nlw} in dimension $3$). The set $\prnl^{stat}$ of such small solutions is also a manifold whose tangent set at $0$ is the set $P_L(R)^{stat}$ of linear solutions  of $\Delta u_L=0$; but $\prnl^{stat}$ is actually a strict subset of $\prnl$, by a substantial margin: see Remark \ref{rm:1} for more precisions. 

$P_L(R)^{stat}$ is also a subset of $P_L(R)$, so we recover the inclusion $\prnl^{stat}$ into $\prnl$ at the tangent space level. Yet, in \cite{CL:24}, we give a more precise statement: the nonlinear static solutions of $\prnl^{stat}$ ``look'' like the linear one $P_L(R)^{stat}$ at infinity. In a suitable space $Z_{r}$ of analytic functions on $\mathbb{S}^{d-1}$ adapted to the operator $\Delta$, there exists a unique $u_L\in P_L(R)^{stat}$ so that  
\begin{align*}
\nor{(u-u_{L})(r\cdot)}{Z_{r}}\underset{r\to +\infty}{\longrightarrow}0.
\end{align*}
Moreover, the application $u\mapsto u_L$, that appears as a kind of scattering operator is a (local) bijection. It would be very interesting to obtain such precise description for the nonlinear non radiative solutions. 

\bigskip

Our second result is related to wave operator: it says that given any radiation field $F$ (as in \eqref{eq:expansion_w_1}), there exists a unique nonlinear solution of \eqref{eq:nlw} with this prescribed radiation field. This result already appears in the literature, in a form or another, see below. The precise statement is as follows.

\begin{thm} \label{thm:2}
Let $3\le d\le 6$ and $F\in L^{2}(\R \times  \Sdmu)$. Then, there exist $T \in \m R$ and a unique  $u \in X([T,+\infty)$ solution of the nonlinear equation \eqref{eq:nlw} so that, as $t \to +\infty$,
\begin{gather} \label{eq:expansion_w_2}
\nabla_{t,x} u(t,x) - \frac{1}{\sqrt{2}|x|^{\frac{d-1}{2}}}F \left(|x|-t, \frac{x}{|x|} \right) \times \begin{pmatrix}
-1 \\
x/|x|
\end{pmatrix}  \to 0 \quad \text{in} \quad L^2(\m R^d)^{1+d}.
\end{gather}
Furthermore, if $\| F \|_{L^{2}(\R \times  \Sdmu)}$ is small enough, one can choose $T=0$ and $u \in X(\m R)$ is defined globally.
\end{thm}

We refer to \cite[Theorem 1.1]{LS23} for a result with a similar flavor, for wave type equations (with other nonlinearities) in dimension $3$, but in different functional spaces; see also \cite{LS24}. 

Theorem \ref{thm:2} is independent of Theorem \ref{thm:1}, but relies on a linear scattering result in $X$ and on Proposition \ref{propbijprofil} which ensures that the map giving the radiation of (from a linear solution) is onto: this last result goes back to Friedlander \cite{F:80} (see also the appendix of \cite{DKM:19}). Here, we provide a new proof based on formula \eqref{def:T}, which actually gives an explicit expression of the inverse.

\section{Proofs}

The spaces $W(I)$, $X(I)$ and $N(I)$ were chosen to satisfy the following Strichartz and nonlinear estimates. For a constant independent of the interval $I$ (or of $t \in \m R$), we have
\begin{align}
\nor{S_L(t)(u_{0},u_{1})}{X(\R)}&\le C \nor{(u_{0},u_{1})}{\mathcal{H}}, \\
\nor{(u(0),\partial_t u(0))}{\mathcal{H}} & \le C\nor{u}{X(\R)}, \\
\nor{\int_{-\infty}^{+\infty}\cos(\tau |D_{x}|) h(\tau) d\tau}{L^{2}(\m R^{d})}&\le C \nor{h}{N(\R)}, \\
\nor{\int_{-\infty}^{+\infty}\cos(\tau |D_{x}|) h(\tau) d\tau}{L^{2}(\m R^{d})}&\le C \nor{h}{N(\R)}, \\
\label{Strichartz} \nor{\int_{-\infty}^{+\infty}\sin(\tau |D_{x}|) h(\tau) d\tau}{L^{2}(\m R^{d})}& \le C \nor{h}{N(\R)}, \\
\nor{v}{X([t,+\infty))} + \nor{\vec v(t)}{\q H}  &  \le C  \nor{h}{N([t,+\infty))} \\
\text{where} \quad v(t) = \int_{t}^{+\infty}\frac{\sin((.-\tau)|D_{x}|)}{|D_{x}|} & h(\tau) d\tau.
\end{align}
The related Strichartz estimates can for example be found in \cite[Theorem 3.1]{LS:95}, see also \cite{GV:95}.
Also notice that $N$ is such that if $h \in N([A,+\infty))$ for some $A \in \m R$, then 
\begin{align}
\label{limiteN}
    \| h \|_{N([t,+\infty))} \to 0\textnormal{ as } t \to +\infty,
\end{align}
(and similarly in a neighbourhood of $-\infty$). We will finally need the nonlinear estimate
\begin{align}
\label{estimNfX}
\nor{f(u)-f(v)}{N(I)}&\le C\nor{u-v}{W(I)}(\nor{u}{W(I)}^{q-1}+\nor{v}{W(I)}^{q-1}).
\end{align}
It does hold in the cases considered for \eqref{eq:nlw} since $|f'(s)| \le C |s|^{q-1}$ and due to Hölder estimates. In fact, our proofs work in any functional setting that respects the above conditions \eqref{Strichartz}-\eqref{limiteN}-\eqref{estimNfX}.

\bigskip

Let us start by a a few observations related to the operator $\q T$.

\begin{definition}
We denote:
\begin{align*}
L^{2}_{\textnormal{odd}}(\R \times  \Sdmu) & :=\left\{ F \in L^{2}(\R \times \Sdmu); \ F(s,\omega)=-F(-s,-\omega), a.e. \right\}, \\
L^{2}_{\textnormal{even}}(\R \times \mathbb{S}^{d-1}) & :=\left\{ F \in L^{2}(\R \times \Sdmu); \ F(s,\omega) = F(-s,-\omega), a.e. \right\}.
\end{align*}
\end{definition}

\begin{lem}[{\cite[Lemma 4.14]{CL24}}] \label{lem:range} Let $d$ be odd.
$\Td$ defines an isometry from $L^{2}(\m R^{d}) \to L^{2}(\R \times \Sdmu)$ and is therefore an isomorphism to its range defined by
\[ Range(\Td) = \begin{cases}
L^{2}_{\textnormal{even}}(\R \times  \Sdmu) & \text{ if } d \equiv 1 [4], \\
L^{2}_{\textnormal{odd}}(\R \times  \Sdmu)& \text{ if } d \equiv 3 [4].
\end{cases} \]
Similarly, $\partial_s \Td: \dot H^1 (\m R^{d}) \to L^{2}( \Sdmu\times \R)$ is isometric and
\[ Range(\partial_s \Td) = \begin{cases}
L^{2}_{\textnormal{odd}}(\R \times  \Sdmu) & \text{ if } d \equiv 1 [4], \\
L^{2}_{\textnormal{even}}(\R \times  \Sdmu)& \text{ if } d \equiv 3 [4].
\end{cases}. \]
\end{lem}

We obtain the following corollary.
\begin{cor}\label{cor:inverTd}Let $d$ be odd and $R>0$. There exists a continuous linear map $G_{R}^{1}:L^{2}((R,+\infty) \times  \Sdmu)\mapsto L^{2}(\m R^{d})$ so that for any $F\in L^{2}((R,+\infty) \times  \Sdmu)$, $\Td  G_{R}^{1} F=F$ a.e. on $(R,+\infty)$.

Similarly, there exists a continuous linear map $G_{R}^{0}:L^{2}((R,+\infty) \times  \Sdmu)\mapsto \dot{H}^{1}(\m R^{d})$ so that for any $F\in L^{2}((R,+\infty) \times  \Sdmu)$, $\partial_{s}\Td  G_{R}^{0} F=F$ a.e. on $(R,+\infty)$.
\end{cor}

\begin{proof}
We just prove the result for $\Td$ and $d \equiv 1 [4]$, the other cases being similar. Since $ Range(\Td)=L^{2}_{\textnormal{even}}(\R \times  \Sdmu)$ is a closed subsets of the Banach space $L^{2}(\R \times \Sdmu)$, we can apply the open mapping Theorem of Banach to define a continuous inverse $\Td^{-1}$ from $L^{2}_{\textnormal{even}}(\R \times  \Sdmu)$ to $L^{2}(\m R^{d})$. Let $\widetilde{F}$ be the even extension of $F\in L^{2}((R,+\infty) \times  \Sdmu)$ that is equal to zero on $s\in [-R,R]$. More precisely
\begin{align*}
\widetilde{F}(s,\omega)&=F(s,\omega)\quad \textnormal{ for }s>R\\
\widetilde{F}(s,\omega)&=F(-s,-\omega)\quad \textnormal{ for }s<-R\\
\widetilde{F}(s,\omega)&=0\quad \textnormal{ for }s\in [-R,R]. 
\end{align*}
It is clear that $\widetilde{F}\in L^{2}_{\textnormal{even}}(\R \times  \Sdmu)$. Defining $G_{R}^{1}F=\Td^{-1}\widetilde{F}$, we obtain $\Td  G_{R}^{1}F =\widetilde{F}$ which satisfies the expected result.
\end{proof}

Given a source term $f$, we can now construct a solution to the linear equation with this source term, which is non radiative.

\begin{prop}\label{lmDuhPi}Let $d$ odd and $X$, $N$ functional spaces satisfying \eqref{Strichartz} and \eqref{limiteN}.
Let $h\in N(\R)$. There exists a \emph{continuous} linear map $T:  N(\m R) \to X(\m R)$, such that for any $h \in N(\m R)$, $u = Th$ is the unique element $u\in X(\R)$ satisfying
\begin{enumerate}
\item $u$ is solution of $\Box u=h$,
\item $E_{\ext,R}(\vec u)=0$,
\item \label{propPiR0}$\pi_{R}(\vec u(0))=0$.
\end{enumerate}
\end{prop}

\begin{proof}

\emph{Step 1.} We first look for $\widetilde{u}$ satisfying the hypothesis 1) and 2), but not necessarily 3) We decompose $\tilde u = v+ w$ where
\[ v(t) :=\int_{-\infty}^{t}\frac{\sin((t-\tau)|D_{x}|)}{|D_{x}|}h(\tau) d\tau, \]
so that $\Box v=h$ with morally $0$ data at $-\infty$ and $w$ solution of $\Box w=0$ is to be chosen later on. Notice that changing $t$ to $-t$ in \eqref{Strichartz}, we get
\begin{align*}
\nor{\vec{v}(t)}{\mathcal{H}}\le\nor{\int_{-\infty}^{t}\frac{\sin((t-\tau)|D_{x}|)}{|D_{x}|}h(\tau) d\tau}{\q H}&\le C \nor{h}{N((-\infty,t])}.
\end{align*}
Using \eqref{limiteN}, this  directly implies
\begin{align} \label{limv-infty}
 \lim_{t \to -\infty} ( \nor{\nabla v}{L^2(\m R^{d})}^2 +  \nor{\partial_t v}{L^2(\m R^{d})}^2)=0.
\end{align} 
Also, by  \eqref{Strichartz}, there hold
\[ \nor{v}{X(\R)}\le C \nor{h}{N(\R)}. \]Let us now estimate the exterior energy (outside a truncated cone) of $\vec v$ as  $t\to +\infty$. We write
\begin{align}
v(t)&=\int_{-\infty}^{+\infty}\frac{\sin((t-\tau)|D_{x}|)}{|D_{x}|}h(\tau) d\tau-\int_{t}^{+\infty}\frac{\sin((t-\tau)|D_{x}|)}{|D_{x}|}h(\tau) d\tau\\
&=\frac{\sin(t|D_{x}|)}{|D_{x}|}\int_{-\infty}^{+\infty}\cos(\tau |D_{x}|)h(\tau) d\tau-\cos(t |D_{x}|)\int_{-\infty}^{+\infty}\frac{\sin(\tau|D_{x}|)}{|D_{x}|}h(\tau) d\tau  \\
& \qquad -\int_{t}^{+\infty}\frac{\sin((t-\tau)|D_{x}|)}{|D_{x}|}h(\tau) d\tau  \\
\label{defv} &=: \frac{\sin(t|D_{x}|)}{|D_{x}|}v_{1+}+\cos(t |D_{x}|)v_{0+}+r(t). \end{align}
In other words, $\vec v = S_L(v_{0+},v_{1+})+ \vec r$.
We estimate using \eqref{Strichartz}
\begin{align}
\label{e:v+r}
\nor{v_{1+}}{L^{2}(\m R^{d})}&=\nor{\int_{-\infty}^{+\infty}\cos(\tau |D_{x}|)h(\tau) d\tau}{L^{2}(\m R^{d})}\le C \nor{h}{N(\R)}, \\
\nor{v_{0+}}{\dot{H}^{1}(\m R^{d})}&=\nor{\int_{-\infty}^{+\infty}\sin(\tau |D_{x}|)h(\tau) d\tau}{L^{2}(\m R^{d})}\le C \nor{h}{N(\R)}, \\
\nor{r}{X(\R)}&=\nor{\int_{.}^{+\infty}\frac{\sin((\cdot-\tau)|D_{x}|)}{|D_{x}|}h(\tau) d\tau}{X(\R)}\le C \nor{h}{N(\R)}, \\
\nor{\vec r(t)}{\mathcal{H}}&=\nor{\int_{t}^{+\infty}\frac{\sin((t-\tau)|D_{x}|)}{|D_{x}|} h(\tau) d\tau}{\mathcal{H}}\le C \nor{h}{N([t,+\infty))}.
\end{align}
In particular, due to \eqref{limiteN}, we have
\begin{align}
\label{limr}
 \lim_{t \to +\infty} ( \nor{\nabla r}{L^2(\m R^{d})}^2 +  \nor{\partial_t r}{L^2(\m R^{d})}^2)=0.
 \end{align}
We will select $(w_{0},w_{1}) = (w(0),\partial_{t} w(0))$ the data at initial time for $w$, so that $w(t) = S_L(t) (w_0,w_1)$.  We can now compute the radiation of $u$ in terms of $(w_0,w_1)$ and $(v_{0+}, v_{1+})$. Indeed, for $t \to -\infty$, using \eqref{equivfinallemmaH1_2-infty} and \eqref{limv-infty}, we have
\begin{align*}
\MoveEqLeft \lim_{t \to -\infty} (\nor{\nabla \widetilde{u}}{L^2(|x| \ge |t|+R)}^2 +  \nor{\partial_t \widetilde{u}}{L^2(|x| \ge |t|+R)}^2) \\
& =\lim_{t \to -\infty} (\nor{\nabla w(t)}{L^2(|x| \ge |t|+R)}^2 +  \nor{\partial_t w(t)}{L^2(|x| \ge |t|+R)}^2)\\
&=\|  \partial_{s}\Td w_{0} + \Td w_{1} \|_{L^2([R,+\infty) \times \Sdmu)}^2. 
\end{align*}
Similarly, for $t \to +\infty$, $\vec{\tilde u}(t) = S_L(t)(w_0+v_{0+},w_1+v_{1+}) + \vec r(t)$ so that  using \eqref{limr} and \eqref{equivfinallemmaH1_2}, we have
\begin{align*}
\MoveEqLeft\lim_{t \to + \infty} (\nor{\nabla \widetilde{u}}{L^2(|x| \ge |t|+R)}^2 +  \nor{\partial_t \widetilde{u}}{L^2(|x| \ge |t|+R)}^2) \\
& =\lim_{t \to +\infty} (\nor{\nabla (v + w)(t)}{L^2(|x| \ge |t|+R)}^2 +  \nor{\partial_t (v + w)(t) }{L^2(|x| \ge |t|+R)}^2)\\
&= \|  \partial_{s}\Td (w_{0}+v_{0+}) - \Td (w_{1}+v_{1+}) \|_{L^2([R,+\infty) \times \Sdmu)}^2.
\end{align*}
Hence, summing up, we get:
\begin{multline} \label{eq:ext_tildeu}
E_{\ext,R}(\vec{\tilde u})  = \frac{1}{2} \| \partial_s \q T w_0 + \q T w_1 \|_{L^2([R,+\infty) \times \m S^{d-1})} \\
 + \frac{1}{2} \| \partial_s \q T (w_0+v_{0+}) - \q T (w_1+v_{1+}) \|_{L^2([R,+\infty) \times \m S^{d-1})}.
\end{multline}
We therefore look for $(w_{0},w_{1})\in \dot{H}^{1}\times L^{2}$ such that 
\begin{align}
\begin{cases} \partial_{s}\Td w_{0}+ \Td w_{1}=0 \text{ on } [R,+\infty)\times \mathbb{S}^{d-1}, a.e.\\
\partial_{s}\Td (w_{0}+v_{0+})- \Td (w_{1}+v_{1+})=0 \text{ on } [R,+\infty)\times \mathbb{S}^{d-1}, a.e. \label{eqnw01+}
\end{cases}
\end{align}
Equivalently: 
\[
\begin{cases}
2\Td w_{1}=-\Td v_{1+}+\partial_{s}\Td v_{0+} \text{ on } [R,+\infty)\times \mathbb{S}^{d-1}, a.e. \\
2\partial_{s}\Td w_{0}=-\partial_{s}\Td v_{0+}+ \Td v_{1+} \text{ on } [R,+\infty)\times \mathbb{S}^{d-1}, a.e.
\end{cases}
\]
Due to Corollary \ref{cor:inverTd}, the previous equations can be solved with a continuous inverse. 

To summarize, we finally define
\begin{align}
\label{defw10}
w_{1}=\frac{1}{2}G_{R}^{1}(-\Td v_{1+}+\partial_{s}\Td v_{0+}) \quad \text{and} \quad w_{0}=\frac{1}{2}G_{R}^{0}(-\partial_{s}\Td v_{0+}+\Td v_{1+}).
\end{align}
Then  $(w_{0},w_{1})$ solve the system \eqref{eqnw01+} and, thanks to Lemma \ref{lem:range} and \eqref{e:v+r}, satisfy the estimates
\begin{align*}
\nor{(w_{0}, w_{1})}{\mathcal{H}} & \le C\nor{\Td v_{1+}-\partial_{s}\Td v_{0+}}{L^{2}([R,+\infty) \times  \Sdmu)} \\
& \le C\nor{v_{1+}}{L^{2}(\m R^{d})}+C\nor{v_{0+}}{\dot{H}^{1}(\m R^{d})}\le C \nor{h}{N(\R)}.
\end{align*}
Then we let $\vec{\widetilde{u}}= \vec v+ S_L (w_0,w_1)$ where $\vec v$ is defined in \eqref{defv} and $(w_{0}, w_{1})$ is defined in \eqref{defw10}. Then  $\Box \widetilde{u}=\Box v=h$ and, in view of \eqref{eq:ext_tildeu}, $E_{\ext,R}(\vec{\tilde u}) =0$. Also, we have the bound
\begin{align}
\nor{\widetilde{u}}{X(\R)}\le \nor{v}{X(\R)}+\nor{w}{X(\R)}\le C\nor{h}{N(\R)}+\nor{(w_{0}, w_{1})}{\mathcal{H}}\le C\nor{h}{N(\R)}.
\end{align}

\bigskip

\emph{Step 2.} Now that $\tilde u$ is defined, we simply let $u=\widetilde{u}-u_{R}$ where $\vec u_{R}= S_L(\pi_{R}\left(\widetilde{u}(0),\partial_{t}\widetilde{u}(0)\right))$: indeed, $u_R$ is a non radiative solution, and solves $\Box u_{R}=0$. Also, regarding continuity of the map, we just need to write
\begin{align}
\nor{u_{R}}{X(\R)}\le C\nor{\pi_{R}\left(\widetilde{u}(0),\partial_{t}\widetilde{u}(0)\right)}{\mathcal{H}}\le \nor{\left(\widetilde{u}(0),\partial_{t}\widetilde{u}(0)\right)}{\mathcal{H}}\le C\nor{\widetilde{u}}{X(\R)}\le C\nor{h}{N(\R)},
\end{align}
so that $\| u \|_{X(\m R)} \le C \|h \|_{N(\m R)}$. This finishes the existence part.

\bigskip

\emph{Step 3.} Concerning uniqueness: let $u_{1}$ and $u_{2}$ be two such solutions of the problem. In particular, $z=u_{1}-u_{2}$ satisfy:
\begin{enumerate}
\item $z$ is solution of $\Box z=0$,
\item $E_{\ext,R}(\vec z)=0$,
\item $\pi_{R}(\vec z(0))=0$.
\end{enumerate}
In particular, the first and second assumptions imply $(z(0),\partial_{t} z(0))\in P_L(R)$ and therefore $\vec z(0)=\pi_{R}(\vec z(0))$. Together with the third assumption, we infer $(z(0),\partial_{t}z(0))=0$ and therefore $z=0$, and $u_1=u_2$.
\end{proof}

With Proposition \ref{lmDuhPi} in hand, we can now prove the theorem.

\begin{proof}[Proof of Theorem \ref{thm:1}]
For $(u_{0},u_{1})\in \q H$, let $\vec u_L = S_L(u_0,u_1)$. We are looking for a solution $u$ of 
\begin{align} \label{eq:G}
u= u_L + T(f(u)).
\end{align}
Indeed, if $u \in X(\m R)$ solves \eqref{eq:G}, then
\[ \Box u = \Box (T f(u)) = f(u), \]
so that $u $ solves \eqref{eq:nlw}. 
To solve \eqref{eq:G}, given $(u_0,u_1) \in P_L(R)$ with $\| (u_0,u_1) \|_{\q H} \le \e$, we use a fixed point argument on small closed balls $B(0,\e)$ of $X(\m R)$ for the map
\[ G: r \mapsto  T(f(u_L+r)). \]
Due to the continuity of $T: N(\m R)\to X(\m R)  $ (provided by Proposition \ref{lmDuhPi}), and using \eqref{Strichartz} and \eqref{estimNfX}, we get for $r, \tilde r \in X(\m R)$,
\begin{align*}
\MoveEqLeft \nor{G(r)}{X(\R)}\le C \nor{f(u_{L}+r)}{N(\R)}\le C \nor{u_{L}+r}{X(\R)}^{q}\le C (\e^{q}+\nor{r}{X}^{q}), \\
\MoveEqLeft \nor{G(r)-G(r'))}{X(\R)} \le C \nor{f(u_{L}+r)-f(u_{L}+\widetilde{r})}{N(\R)}\\
& \le C \nor{r-\widetilde{r}}{W(\R)}(\nor{u_{L}+r}{W(\R)}^{q-1}+\nor{u_{L}+\widetilde{r}}{W(\R)}^{q-1}) \\
& \le C \nor{r-\widetilde{r}}{X}(\e^{q-1}+\nor{r}{X}^{q-1}+\nor{\widetilde{r}}{X}^{q-1}) .
\end{align*}
So, for $\e$ small enough, $G$ admits a unique fixed point $v$ in $\overline B_{X(\m R)}(0,\e)$, the closed ball of radius $\e$ in $X(\m R)$. Furthermore 
\begin{equation} \label{est:fp_G}
\| v \|_{X(\m R)} = \| G(v) \|_{X(\m R)} \le C \| (u_0,u_1) \|_{\q H}^q.
\end{equation}
Then $u := u_L + v$ solves \eqref{eq:G}. Also, by regularity of the Banach fixed point with parameter, the map $(u_0,u_1) \mapsto v$ is $\mathscr C^{1}$ from $B_{\q H}(0,\e)$ to $X(\m R)$ (notice that the nonlinearity is $\mathscr C^{1}$), with differential $0$ at $0 \in \q H$, due to \eqref{est:fp_G}. 
 Finally,
  
\[  \pi_{R}(u(0),\partial_{t}u(0))=\pi_{R}(u_{0}, u_{1}) + \pi_R(T(f(u))(0), \partial_{t}T(f(u))(0)) =\pi_{R} (u_{0}, u_{1}). \]
Therefore, the map
\[ \Phi: (u_0,u_1) \mapsto (u,\partial_t u)(0), \]
(where $u$ is as above) satisfies the first part of Theorem \ref{thm:1}, up to possibly diminishing $\e$.

Now, assuming $(u_{0},u_{1})\in P_L(R)$, we define $\vec u = \q S\Phi (u_{0},u_{1})$, the associated nonlinear solution. We have $f(u) \in N(\m R)$ due to \eqref{estimNfX} and as $E_{\ext,R}(S_L(u_{0},u_{1})) =0$, the radiation of $u$ is well defined and
\[ E_{\ext,R}(\vec u) =  E_{\ext,R}(T(f(u)), \partial_t T(f(u))) =0, \]
so that $u \in \prnl$. So, we have proved $\Phi(P_L(R)\cap B_{\q H}(0,\e))\subset \prnl\cap V$. 

Reciprocally, let $(v_{0},v_{1})\in  \prnl\cap V$. By definition of $V$, it can be written $(v_{0},v_{1})=\Phi (u_{0},u_{1})$ with $(u_{0},u_{1})\in B_{\q H}(0,\e)$. Denoting $\vec u = \q S\Phi (u_{0},u_{1})=\q S (v_{0},v_{1})$, the associated nonlinear solution, we have, by definition of $\Phi$, $\vec u=S_L (u_{0},u_{1}) + T(f(u))$. In particular, as $E_{\ext,R}(T(f(u)), \partial_t T(f(u)))=0$, we have
\[ E_{\ext,R}(\vec u) = E_{\ext,R}(S_L (u_{0},u_{1})). \]
Now we assumed $(v_{0},v_{1})\in \prnl$, so that $E_{\ext,R}(\vec u) =E_{\ext,R}(\q S (v_{0},v_{1}))=0$, and
\[ E_{\ext,R}(S_L (u_{0},u_{1}))=0. \]
Thus, $(u_{0},u_{1})\in P_L(R)$ and $(v_{0},v_{1})\in \Phi (P_L(R)\cap B_{\q H}(0,\e))$. 

The last statement of the theorem is only a rephrasing of the previous results in terms of submanifolds in Banach spaces.
\end{proof}

Now, we turn to the proof of Theorem \ref{thm:2} and begin by a Proposition stating that the radiation operator is onto. 

\begin{prop}[Friedlander \cite{F:80}]
\label{propbijprofil}
The application 
\begin{align*}
   \q H &\longrightarrow L^2( \m R \times \m S^{d-1})\\
   (v_{0},v_{1})&\longmapsto \partial_{s}\Td v_{0} - \Td v_{1}
\end{align*}
is a bijective isometry.
\end{prop}

\begin{proof} For the convenience of the reader, we provide a proof with an explicit inversion formula in terms of Fourier transform. Formula \eqref{def:T} gives 
\begin{align}
\q F_{s \to \nu}  (\partial_s \Td v_0-\Td v_1)(\nu, \omega)
 =  c_0  |\nu|^{\frac{d-1}{2}} (e^{i \tau}  \m 1 _{\nu < 0}   + e^{-i\tau} \m 1_{\nu \ge 0} ) (i \nu \hat v_0(\nu \omega) -\hat v_1(\nu \omega)). 
\end{align}
For the injectivity, we could compute directly that the application is an isometry, see for instance \cite[Lemma 2.1.]{CL24} for a closely related computation. Here we can directly check that $\partial_{s}\Td v_{0} - \Td v_{1}=0$ implies $i \nu \hat v_0(\nu \omega) =\hat v_1(\nu \omega)$ almost everywhere in $\m R \times \m S^{d-1}$. Applying at $(\nu,\omega)$ and $(-\nu,-\omega)$, it gives $(v_0,v_1)=(0,0)$.

For the surjectivity, given $F \in L^2(\m R \times \m S^{d-1})$, denote for simplicity $\hat F = \q F_{s \to \nu}F$, and define $v_0$ and $v_1$ by their Fourier transform as follows: for $\xi \in \m R^d \setminus \{ 0 \}$, with $\xi = \rho \omega$ where $\rho > 0$ and $\omega \in \m S^{d-1}$, we set
\begin{align*}
    \hat v_0 (\xi) & = \frac{1}{2i c_0 \rho^{\frac{d+1}{2}}} \left( e^{i \tau} \hat F(\rho, \omega) - e^{-i \tau} \hat F(-\rho,-\omega) \right), \\
     \hat v_1 (\xi) & = -\frac{1}{2 c_0 \rho^{\frac{d-1}{2}}} \left( e^{i \tau} \hat F(\rho, \omega) + e^{-i \tau} \hat F(-\rho,-\omega) \right).
\end{align*}
Then for $\omega \in \m S^{d-1}$, we have for $\nu >0 $
\begin{align*}
\MoveEqLeft 
\q F_{s \to \nu}  (\partial_s \Td v_0-\Td v_1)(\nu, \omega)
 =  c_0  \nu^{\frac{d-1}{2}} e^{-i\tau} (i \nu \hat v_0(\nu \omega) -\hat v_1(\nu \omega)) \\
& = c_0 \nu^{\frac{d-1}{2}} e^{-i\tau} \left(  \frac{i \nu}{2i c_0 \nu^{\frac{d+1}{2}}} \left( e^{i \tau} \hat F(\nu, \omega) - e^{-i \tau} \hat F(-\nu,-\omega) \right) \right. \\
& \qquad \qquad \left. + \frac{1}{2 c_0 \nu^{\frac{d-1}{2}}} \left( e^{i \tau} \hat F(\nu, \omega) + e^{-i \tau} \hat F(-\nu,-\omega) \right) \right) \\
& = \hat F(\nu,\omega),
\end{align*}
and if $\nu <0$,
\begin{align*}
\MoveEqLeft 
\q F_{s \to \nu}  (\partial_s \Td v_0-\Td v_1)(\nu, \omega)
=  c_0  |\nu|^{\frac{d-1}{2}} e^{i \tau}   \left( -i |\nu| \hat v_0(|\nu| (- \omega)) -\hat v_1(|\nu| (- \omega)) \right) \\
& = c_0  |\nu|^{\frac{d-1}{2}} e^{i \tau}   \left( - \frac{i |\nu|}{2i c_0 |\nu|^{\frac{d+1}{2}}} \left( e^{i \tau} \hat F(|\nu|, -\omega) - e^{-i \tau} \hat F(\nu,\omega) \right) \right. \\
& \qquad \qquad \left. + \frac{1}{2 c_0 |\nu|^{\frac{d-1}{2}}} \left( e^{i \tau} \hat F(|\nu|, -\omega) + e^{-i \tau} \hat F(\nu,\omega) \right) \right) \\
& = \hat F(\nu,\omega).
\end{align*}
Hence there hold
\[ (\partial_s \Td v_0-\Td v_1) = F. \]

We verify that $(v_0,v_1)$ defined as above are indeed in $\q H$.
\begin{align*}
\nor{v_0}{\dot{H}^1}^2 & =\frac{1}{(2\pi)^d}\nor{|\cdot|\hat v_0(\cdot)}{L^2}^2=\frac{1}{(2\pi)^d}\int_0^{+\infty}\rho^{d-1}\int_{\omega\in \Sdmu} \rho^2\left|\hat v_0(\rho\omega)\right|^2~d\omega~d\rho\\
&=\frac{1}{4 c_0^2 (2\pi)^d}\int_0^{+\infty}\int_{\omega\in \Sdmu} \left|e^{i \tau} \hat F(\rho, \omega) - e^{-i \tau} \hat F(-\rho,-\omega)  \right|^2~d\omega~d\rho.
\end{align*}
\begin{align*}
\nor{v_1}{L^2}^2&=\frac{1}{(2\pi)^d}\nor{\hat v_1(\cdot)}{L^2}^2=\frac{1}{(2\pi)^d}\int_0^{+\infty}\rho^{d-1}\int_{\omega\in \Sdmu} \left|\hat v_1(\rho\omega)\right|^2~d\omega~d\rho\\
&=\frac{1}{4 c_0^2 (2\pi)^d}\int_0^{+\infty}\int_{\omega\in \Sdmu} \left|e^{i \tau} \hat F(\rho, \omega) + e^{-i \tau} \hat F(-\rho,-\omega)  \right|^2~d\omega~d\rho.
\end{align*}
Finally, it is an isometry: indeed, $\ds \frac{1}{4 c_0^2 (2\pi)^d}=\frac{1}{4\pi}$ and 
\begin{multline*}
 \left|e^{i \tau} \hat F(\rho, \omega) - e^{-i \tau} \hat F(-\rho,-\omega)  \right|^2+ \left|e^{i \tau} \hat F(\rho, \omega) + e^{-i \tau} \hat F(-\rho,-\omega)  \right|^2 \\
 = 2\left|\hat F(\rho, \omega) \right|^2+ 2\left| \hat F(-\rho,-\omega)  \right|^2,
\end{multline*}
so that
\begin{align*}
\MoveEqLeft \nor{v_0}{\dot{H}^1}^2+\nor{v_1}{L^2}^2 =\frac{1}{2\pi}\int_0^{+\infty}\int_{\omega\in \Sdmu}\left|\hat F(\rho, \omega) \right|^2+ \left| \hat F(-\rho,-\omega)  \right|^2~d\omega~d\rho\\
&=\frac{1}{2\pi}\int_0^{+\infty}\int_{\omega\in \Sdmu}\left|\hat F(\rho, \omega) \right|^2+ \left| \hat F(-\rho,\omega)  \right|^2~d\omega~d\rho \\ 
& =\frac{1}{2\pi}\int_{\R}\int_{\omega\in \Sdmu}\left|\hat F(\rho, \omega) \right|^2~d\omega~d\rho =\int_{\R}\int_{\omega\in \Sdmu}\left|F(s, \omega) \right|^2~d\omega~ds \\
& =\nor{F}{L^2(\R\times \Sdmu)}^2. \qedhere
\end{align*}
\end{proof}

\begin{proof}[Proof of Theorem \ref{thm:2}]
\emph{Step 1.} We first construct the linear scattering state, that is find $(v_0,v_1) \in  \q H$ such that, denoting $\vec v_L= S_L(v_0,v_1)$, as $t \to +\infty$,
\begin{gather} \label{eq:expansion_w_3}
\nabla_{t,x} v_L(t,x) - \frac{1}{\sqrt{2}|x|^{\frac{d-1}{2}}}F \left(|x|-t, \frac{x}{|x|} \right) \times \begin{pmatrix}
-1 \\
x/|x|
\end{pmatrix}  \to 0 \quad \text{in} \quad L^2(\m R^d)^{1+d}.
\end{gather}
Due to Proposition \ref{propbijprofil}, there exists $(v_0,v_1)\in \q H$ so that 
\[ F=(\partial_s \Td v_0 - \Td v_1). \]
In view of \eqref{eq:expansion_w_1}, we see that $\vec v_L = S_L(\cdot)(v_0,v_1)$ satisfies the expected asymptotic \eqref{eq:expansion_w_3}.

\bigskip

\emph{Step 2.} We now construct $\vec u$, solution to \eqref{eq:nlw} such that $\| \vec u - \vec v_L(t) \|_{\q H} \to 0$ as $t \to +\infty$: this is simply the wave operator, and is standard. We provide some elements of proof for the sake of completeness. We decompose $\vec u(t) = \vec v_L(t) + \vec w(t)$ and write $\vec w$ as solution of a fixed point problem. Let $T \in \m R$ to be chosen later: the Duhamel formula on $[t, \tau]$ (for $\tau \ge t$) gives 
\[ \vec v_L(\tau) + \vec w(\tau) = S_L(\tau-t) (\vec v_L(t) +\vec w(t))+ \int_t^\tau S_L(\tau-s) \begin{pmatrix} 0 \\ f(v_L(s) + w(s))  \end{pmatrix} ds. \]
Notice that $\vec v_L(t) = S_L(t-T) \vec v_L(T)$; compose by $S_L(t-\tau)$ and let $\tau \to +\infty$: as $\| S_L(t-\tau) \vec w(\tau) \|_{\q H} = \| \vec w(\tau) \|_{\q H}$ is meant to tend to $0$, we arrive at the fixed point formulation:
\[ \vec w(t) = \Psi\vec w (t), \quad \text{where} \quad  \Psi \vec v(t) := - \int_t^{+\infty} S_L(t-s) \begin{pmatrix} 0 \\ f(v_L(s) + v(s)) \end{pmatrix} ds. \]
Let $T \in \m R$ to be fixed later, we work in small closed balls $\overline B(0,\e)$ of $X([T,+\infty))$. By \eqref{Strichartz} and \eqref{estimNfX}, we have for $\vec v \in X([T,+\infty))$,
\[ \| \Psi \vec v \|_{X([T,+\infty))} \le C \|  f(v_L + v) \|_{N([T,+\infty))} \le C \left( \| v_L \|_{W([T,+\infty))}^q + \| v \|_{W([T,+\infty))}^q \right). \]
Similarly,
\begin{align*}
\MoveEqLeft \| \Psi \vec v - \Psi \vec {\tilde v} \|_{X([T,+\infty))} \le C \| f(v_L + v)-f(v_L + \tilde v)  \|_{N([T,+\infty)} \\
 & \le C \left( \| v_L \|_{W([T,+\infty))}^{q-1}  + \| v \|_{W([T,+\infty))}^{q-1} +  \| \tilde v \|_{W([T,+\infty))}^{q-1} \right) \| v - \tilde v \|_{W([T,+\infty))}.
\end{align*}
Let $T$ be such that $\| v_L \|_{W([T,+\infty))}^{q-1} \le \e$ be small enough, then $\Psi$ admits a unique fixed point $\vec w$ in $B(0,\e)$, and $\vec u = \vec v_L + \vec w$ answers the question.
\end{proof}

\appendix

\section{Description of the set \texorpdfstring{$P_L(R)$}{P_L(R)} of linear non radiative solutions} \label{appendix1}

In this section, we gather some results of \cite{CL24} where a precise description of the set $P_L(R)$ was performed for $R>0$. This corresponds to classifying the linear solutions $u$ that have vanishing asymptotic energy on the exterior light cone $|x| \ge t+R$ with $R> 0$, that is
\[ E_{ext,R}(u) =0. \]

By finite speed of propagation, initial data which are compactly supported in  $|x| \le R$ obviously satisfy this condition. We will call this space
\[
\mathcal{K}_{R,comp}= \left\{(u_0,u_1)\in\dot H^1 \times L^2 (\m R^d): (u_0,u_1) |_{\{ |x| > R \}} =0\right\}.
\]
where the equality is in the distributional sense.

It turns out that these are not the only examples. We will now need some further notation.

We denote $(Y_{\ell})_{\ell \in \m M}$
a countable orthonormal basis of spherical harmonics of $\m S^{d-1}$. $Y_{\ell}$ is the restriction to $\m S^{d-1}$ of a harmonic (homogeneous) polynomial. For short, we will denote $l=l(\ell)$ the degree of this polynomial.

The non radiative functions will be the following. Denote for $k \in \m N$,
\[ \alpha_{k}: = -l-d+2k+2. \]
$\alpha_k$ also depends on $\ell$, but here and below, we silence this dependence to keep notations light. Then let
\begin{align} \label{def:gk}
g_{k}(x)= \m 1_{\{|x|>R\}}|x|^{\alpha_{k}}Y_{\ell}\left(\frac{x}{|x|}\right).
\end{align}
Note that $g_k \in L^2$ if and only if $\alpha_k <-d/2$. We introduce
\[
\Nd_{R, \ell}^{0}=\Span \left( g_{k} ;  \text{ for }  k \in \m N \text{ such that }\alpha_{k}<-d/2\right)
\]
Similarly, let
\begin{align} \label{def:fk}
f_{k}(x)= \begin{cases}
\ds  \left( \frac{|x|}{R} \right)^{\alpha_{k}}Y_{\ell}\left(\frac{x}{|x|}\right) & \text{ for }|x|>R\\
\ds  \left( \frac{|x|}{R} \right)^{l}Y_{\ell}\left(\frac{x}{|x|}\right) & \text{ for }|x|\le R.
\end{cases}
\end{align}
Note that $f_k \in \dot H^{1}$ if and only if $\alpha_{k}<-d/2+1$. Also, the value of $f_{k}$ in $|x|\le R$ is not very important; our choice allows to keep continuity and that the restriction $f_{k}|_{\left\{|x| < R\right\}}$ is a harmonic polynomial, so that $f_{k}$ is orthogonal to (in $\dot H^1$) to functions with compact support in $B(0,R)$.

Let
\[ \Nd_{R, \ell}^{1}=\Span \left( f_{k};  \text{ for } k \in \m N \text{ such that } \alpha_{k}<-d/2+1\right). \]
For any $\ell\in \M$, we note the space
\[ 
P_{\ell}(R)=\Nd_{R, \ell}^{0} \times \Nd_{R, \ell}^{1}.  \]

\begin{remark} \label{rm:1}
For a fixed spherical harmonics $Y_\ell$, only the value $k=0$ corresponding to $\alpha_0=-l-d+2$ produces a solution of the stationary equation $\Delta u=0$, and from \cite{CL:24} (in dimension 3), a nonlinear stationary solution defined outside a large ball: via time invariance, this yields a curve (manifold of dimension 1) of solutions stationary outside a light cone.

Theorem \ref{thm:1} constructs a non radiative solution for all elements in $P_{\ell}(R)$, which, except for those on the curve above, are \emph{not} stationary outside a light cone.
\end{remark}

One of the result of \cite[Theorem 1.7]{CL24} was the precise description of $P_L(R)$ in odd dimensions as follows. 
\begin{align}
\label{PRHarmonic}
P_L(R)= \mathcal{K}_{R,comp} \stackrel{\perp}{\oplus}  \bigoplus_{\ell\in \m M}^\perp P_{\ell}(R).    
\end{align}
(the orthogonality is related to the natural scalar product of $\dot H^1 \times L^2$).


\bibliographystyle{plain} 
\bibliography{biblio}

\end{document}